\documentclass{amsart}

\usepackage{color,graphicx,amssymb,latexsym,amsfonts,txfonts,amsmath,amsthm}
\usepackage{pdfsync}
\usepackage{amsmath,amscd}
\usepackage[all,cmtip]{xy}

\usepackage{hyperref}
\hypersetup{
    colorlinks=true,       
    linkcolor=blue,          
    citecolor=blue,        
    filecolor=blue,      
    urlcolor=blue           
}

\input epsf
\newcommand{\s}{\vspace{0.3cm}}

\newtheorem{theo}{Theorem} 

\newtheorem{coro}[theo]{Corollary}
\newtheorem{lemm}[theo]{Lemma}

\theoremstyle{remark}
\newtheorem{rema}[theo]{\bf Remark}

\begin{document}
\title{The structure of extended function groups}
\author{Rub\'en A. Hidalgo}
\address{Departamento de Matem\'atica y Estad\'{\i}stica, Universidad de La Frontera, Temuco, Chile}
\email{ruben.hidalgo@ufrontera.cl}

 \thanks{ORCID: https://orcid.org/0000-0003-4070-2819}
\thanks{Partially supported by project Fondecyt 1190001}
\subjclass[2010]{ 30F10, 30F40}
\keywords{Kleinian groups, equivariant loop theorem}

\begin{abstract}
A function group is a finitely generated Kleinian group with an invariant connected component of its region of discontinuity.
An extended function group is a finitely generated extended Kleinian group that contains orientation reversing elements and keep invariant a connected components of its region of discontinuity.

An structural decomposition of function groups, in terms of the Klein-Maskit combination theorems, was provided by Maskit in the middle of the 70's. 
One should expect a similar decomposition structure for extended function groups, but it seems not to be stated in the existing literature.
The aim of this paper is to state and procvide a proof of such a decomposition structural picture. 

\end{abstract}

\maketitle

\section{Introduction}
The conformal (respectively, anticonformal) automorphisms of the Riemann sphere $\widehat{\mathbb C}={\mathbb C} \cup \{\infty\}$ are provided by the M\"obius  (respectively, extended M\"obius) transformations, that is, transformations of the form $T(z)=(az+b)/(cz+d)$ (respectively, $L(z)=(a\overline{z}+b)/(c\overline{z}+d)$) where $a,b,c,d \in {\mathbb C}$ are such that $ad-bc=1$. The group of M\"obius transformations ${\mathbb M}$ is isomorphic to the special projective linear group ${\rm PSL}_{2}({\mathbb C})$ and the group of M\"obius and extended M\"obius transformations is $\widehat{\mathbb M}=\langle {\mathbb M},J(z)=\overline{z}\rangle$. 

A {\it Kleinian group} (respectively, an {\it extended Kleinian group}) is a discrete subgroup of ${\mathbb M}$ (respectively, a discrete subgroup of $\widehat{\mathbb M}$ and containing extended M\"obius transformations). The {\it region of discontinuity} of a (extended) Kleinian group $K$ is the locus of points $p \in \widehat{\mathbb C}$ admitting an open neighborhood $p \in U \subset \widehat{\mathbb C}$ such that $k(U) \cap U \neq \emptyset$ only for finitely many elements $k \in K$.
By the definition, the region of discontinuity is an open set (it might be empty). The complement of the region of discontinuity is called the {\it limit set} and it is the place where the dynamics of the group action is chaotic. The history of Kleinian groups can be traced back to Poincar\'e \cite{Poincare0}.

A {\it function group} is a finitely generated Kleinian group (with a non-empty region of discontinuity) admitting an invariant connected component of its region of discontinuity.
Basic examples of function groups are provided by {\it elementary  groups} (Kleinian groups with finite limit set), 
{\it quasifuchsian groups} (function groups whose limit set is a Jordan curve) and {\it totally degenerate groups} (non-elementary finitely generated Kleinian  groups whose region of discontinuity is both connected and simply-connected). In a series of papers, 
Maskit provided the following decomposition structure of function groups, in terms of the Klein-Maskit combination theorems \cite{Maskit:Comb, Maskit:Comb4, Maskit:book}. 

\begin{theo}[Maskit's decomposition of function groups \cite{Maskit:construction, Maskit:function2, Maskit:function3, Maskit:function4}]
Every function group is constructed from elementary groups, quasifuchsian groups and totally degenerate groups by a finite number of applications of the Klein-Maskit combination theorems. Moreover, in the construction, the
amalgamated free products and the HNN-extensions are realized along either (i) a finite cyclic group (including the trivial group) or (ii) a cyclic group generated by an accidental parabolic element.
\end{theo}

An {\it extended function group} is a finitely generated extended Kleinian group with an invariant connected component of its region of discontinuity. Basic examples of extended function groups are the {\it extended elementary groups} (extended Kleinian groups with finite limit set), {\it extended quasifuchsian groups} (finitely generated extended function groups whose limit set is a Jordan curve) and {\it extended totally degenerate groups} (non-elementary extended finitely generated Kleinian groups with connected and simply-connected region of discontinuity).  

We should note that the term ``extended quasifuchsian group" used in this paper is different from the given by other authors in the sense that they refer it to Kleinian groups whose limit set is a Jordan curve and contains elements permuting the two discs bounded by it.

As it is for the case of function groups, one should expect a similar decomposition result for the extended function groups (see Theorem \ref{corolario3}). It seems that such a result is missing in the literature and the purpose of this note is to provide a simple proof, following Maskit's arguments for the function group case.  

\begin{theo}[Decomposition of extended function groups]\label{corolario3}
Every extended function group is constructed from (extended) elementary groups, (extended) quasifuchsian groups and (extended) totally degenerate groups by a finite number of applications of the Klein-Maskit combination theorems. Moreover, in the construction, the
amalgamated free products and the HNN-extensions are realized along either (i) a finite cyclic group (including the trivial group) or (ii) an infinite dihedral group generated by two reflections or (iii) a cyclic group generated by either an accidental parabolic element or by a pseudo-parabolic transformation whose square is accidental parabolic.
\end{theo}

The idea of the proof is the following. Let $K$ be an extended function group, with invariant connected component $\Delta$. Its index two orientation preserving half  $K^{+}=K \cap {\mathbb M}$  is a function group with the same invariant component.  As $K^{+}$ is finitely generated, Selberg's lemma \cite{Selberg} asserts the existence of a torsion free finite index normal subgroup $\Gamma_{1}$ of $K^{+}$ (which is again a function group). Since $K=\langle K^{+},\tau\rangle$, where $\tau^{2} \in K^{+}$, then $\Gamma=\Gamma_{1} \cap \tau \Gamma_{1} \tau^{-1}$ is a finite index torsion free normal subgroup of $K$.
The Ahlfors finiteness theorem \cite{Ahlfors} asserts that $S=\Delta/\Gamma$ is an analytically finite Riemann surface, that is, $S=\widehat{S}-C$, where $\widehat{S}$ is some closed Riemann surface and $C \subset \widehat{S}$ is a finite set of points (it might be empty). 
The finite group $G=K/\Gamma$ is a group of conformal and anticonformal automorphisms of $S$.
Maskit's decomposition of function groups may be applied to $\Gamma$. There are many possible decompositions, but in order to get one which can be used to obtain a decomposition of $K$, we must find one which is in some sense equivariant with respect to $G$. This is solved by Theorem \ref{lifting2} (equivariant theorem for function groups) obtained by Maskit and the author in \cite{H-M:lifting}. This permits us to obtain a first decomposition structural picture (see Theorem \ref{corolario2}). In such a picture, there may appear (extended) B-groups as factors. A {\it B-group} (respectively, an {\it extended B-group}) is a function group (respectively, an extended function group) with a simply-connected invariant component in its region of discontinuity. 
A subtle modification to Maskit's arguments, for the case of B-groups, to deal with these extended B-groups is provided (see Theorem \ref{propo}).

\s

 A. Haas's thesis \cite{Haas} concerns with  uniformizing  groups of conformal and anticonformal automorphisms acting on plane domains. It leads naturally to extended function groups, but it seems that the above decomposition does not follows immediately from it.

\section{Preliminaries}\label{Sec:prelim}

\subsection{Riemann orbifolds}
A {\it Riemann orbifold} ${\mathcal O}$ consists of a (possible non-connected) Riemann surface $S$ (called the {\it underlying Riemann surface of the orbifold}), an isolated collection of points of $S$ (called the {\it cone points of the orbifold}) and associated to each cone point an integer at least $2$ (called the {\it cone order}).  A connected Riemann orbifold is {\it analytically finite} if its underlying (connected) Riemann surface is the complement of a finite number of points of a closed Riemann surface and the number of cone points is also finite. We may think of a Riemann surface as a Riemann orbifold without cone points.
A {\it conformal automorphism} (respectively, {\it anticonformal automorphism})  of the Riemann orbifold ${\mathcal O}$ is a conformal automorphism (respectively, anticonformal) of the underlying Riemann surface $S$ which preserves both its set of cone points together their cone orders (cone points can be permuted but preserving the cone order). We denote by ${\rm Aut}({\mathcal O})$ (respectively, ${\rm Aut}(S)$) the group of conformal/anticonformal automorphisms of ${\mathcal O}$ (respectively, $S$) and by ${\rm Aut}^{+}({\mathcal O})$ (respectively, ${\rm Aut}^{+}(S)$) its subgroup of conformal automorphisms.

\subsection{Kleinian and extended Kleinian groups}
In the following, we recall some facts on (extended) Kleinian groups. A good source on the topic are the calssical books \cite{Maskit:book,MT}. 
Let us start by observing that, if $K_{1}<K_{2}<\widehat{\mathbb M}$ and $K_{1}$ has finite index in $K_{2}$, then both are discrete if one of them is and, in the discreteness case, both have the same region of discontinuity.

Let $K<\widehat{\mathbb M}$ and set $K^{+}:=K \cap {\mathbb M}$. If $K\neq K^{+}$, then $K^{+}$ is called the {\it orientation-preserving half} of $K$ and, in this case, 
$K$ is an extended Kleinian group if and only if $K^{+}$ is a Kleinian group; in which case both have the same region of discontinuity. 
If moreover, $K$ is an extended Kleinian group and $K^{+}$ is a function group, then either: (i) $K$ is an extended function group or (ii) $K^{+}$ is a quasifuchsian group and there is an element of $K-K^{+}$ permuting both discs bounded by the limits set Jordan curve (so $K$ is not an extended function group).

\subsection{Klein-Maskit's decomposition theorems}
Let $K$ be a Kleinian group with region of discontinuity $\Omega$ and let $H$ be a subgroup of $K$ with limit set $\Lambda(H)$. 
A set $X \subset \widehat{\mathbb C}$ is called {\it precisely invariant under $H$ in $K$} if $E(X)=X$, for every $E \in H$, and $T(X) \cap X = \emptyset$, for every $T \in K\setminus H$. 

We will assume $H$ to be either (i) the trivial group, (ii) a finite cyclic group or (iii) an infinite cyclic group generated by a parabolic transformation.
If $H$ is a cyclic subgroup, a {\it precisely invariant disc} $B$ is the interior of a closed topological
disc $\overline{B}$, where $\overline{B}-\Lambda(H) \subset \Omega$ is precisely invariant under $H$ in $K$.

\begin{theo}[Klein-Maskit's combination theorems \cite{Maskit:Comb, Maskit:Comb4}]\label{KMC}
\mbox{}

\noindent
(1) (Amalgamated free products).
For $j=1,2$, let $K_{j}$ be a Kleinian group, let $H \leq K_{1} \cap K_{2}$ be a cyclic subgroup (either trivial, finite or generated by a parabolic transformation), $H \neq K_{j}$, and 
let $B_{j}$ be a precisely invariant disc under $H$ in $K_{j}$. Assume that $B_{1}$ and $B_{2}$ have as a common boundary the simple loop $\Sigma$ and that $B_{1} \cap B_{2}=\emptyset$.
Then $K=\langle K_{1},K_{2}\rangle$  is a Kleinian group isomorphic to the free product of $K_{1}$ and $K_{2}$ amalgamated over $H$, that is, $K=K_{1} *_{H} K_{2}$, and every elliptic or parabolic element of $K$ is conjugated in $K$ to an element of either $K_{1}$ or $K_{2}$.
Moreover, if $K_{1}$ and $K_{2}$ are both geometrically finite, then $K$ is also geometrically finite.

\noindent
(2) (HNN extensions).
Let $K$ be a Kleinian group. For $j=1,2$, let $B_{j}$ be a precisely invariant disc under the cyclic subgroup $H_{j}$ (either trivial, finite or generated by a parabolic) in $K$, let $\Sigma_{j}$ be the boundary loop of $B_{j}$ and assume that $T(\overline{B}_{1}) \cap \overline{B}_{2} =\emptyset$, for every $T \in K$. Let $A$ a loxodromic transformation such that $A(\Sigma_{1})=\Sigma_{2}$, $A(B_{1}) \cap B_{2}=\emptyset$, and $A^{-1} H_{2} A=H_{1}$. Then $K_{A}=\langle K, A\rangle$ is a Kleinian group, isomorphic to the HNN-extension $K*_{\langle A \rangle}$  (that is, every relation in $K_{A}$ is consequence of the realtions in $K$ and the relations $A^{-1}H_{2}A=H_{1}$). If each $H_{j}$, for $j=1,2$, is its own normalization in $K$, then every elliptic or parabolic element of $K_{A}$ is conjugated to some element of $K$. Moreover, if $K$ is geometrically finite, then $K_{A}$ is also geometrically finite.

\end{theo}

\subsection{An equivariant loop theorem for function groups}\label{Sec:equivariant}
Let $K$ be a function group and $\Delta$ be a $K$-invariant connected component of its region of discontinuity. By the Alhfor's finiteness theorem \cite{Ahlfors,Ahlfors2},
the quotient ${\mathcal O}=\Delta/K$ turns out to be an analytically finite Riemann orbifold. Let ${\mathcal B} \subset {\mathcal O}$ be the (finite) collection of the cone points and let ${\mathcal G} \subset {\mathcal O}-{\mathcal B}$ be the collection of loops which lift to loops under the natural regular holomorphic covering $\pi:\Delta^{0} \to {\mathcal O}-{\mathcal B}$, where $\Delta^{0}$ is the open dense subset of $\Delta$ consisting of those points with trivial $K$-stabilizer.
In \cite{Maskit:planarcover}, Maskit proved the existence of a finite subcollection ${\mathcal F} \subset {\mathcal G}$ of pairwise disjoint loops inside ${\mathcal O}-{\mathcal B}$, each one being a finite power of a simple loop, such that the cover $\pi$ is determined as a highest regular planar cover for which the loops in ${\mathcal F}$  lift to loops (such a collection of loops is not unique). The collection ${\mathcal F}$ is called a {\it fundamental system of loops} of the above regular planar covering. Assume that there is a finite group $H<{\rm Aut}({\mathcal O})$ whose elements lifts to automorphisms of $\Delta$ under $\pi$. Then, in 
\cite{H-M:lifting}, Maskit and the author proved that there is a fundamental system of loops ${\mathcal F}$ being equivariant under $H$.

\begin{theo}[Equivariant loop theorem for function groups \cite{H-M:lifting}]  \label{lifting2}
Let $K$ be a function group, with invariant connected component $\Delta$ in its region of discontinuity,  ${\mathcal O}=\Delta/K$ (which is  an analytically finite Riemann orbifold) and let ${\mathcal B}$ the finite set of cone points of ${\mathcal O}$. Let $\pi:\Delta \to {\mathcal O}$ be the natural regular branched regular covering induced by $K$. Let ${\mathcal G}$ be the collection of loops in ${\mathcal O}-{\mathcal B}$ which lift to loops in $\Delta$ under $\pi$. If $H<Aut({\mathcal O})$ lifts to a group of automorphisms of $\Delta$, then there is a finite sub-collection ${\mathcal F} \subset {\mathcal G}$ such that:
\begin{enumerate}
\item ${\mathcal F}$ consists of pairwise disjoint powers of simple loops;
\item ${\mathcal F}$ is $H$-invariant; and
\item every loop in ${\mathcal G}$ is homotopic to the product of finite powers of a finite loops in ${\mathcal F}$.
\end{enumerate}
The collection ${\mathcal F}$ is called a fundamental set of loops for the pair $(K,H)$.
\end{theo}

\begin{rema}
The condition (3) in the above is equivalet to say that ${\mathcal F}$ is a fundamental system of loops for $\pi$. Also, if 
the function group $K$ is torsion-free, then ${\mathcal O}$ is an analytically finite Riemann surface and each of the loops in the finite collection ${\mathcal F}$ turns out to be a simple loop.
\end{rema}

As a consequence of the above, one may write the following equivariant result for Kleinian groups.

\begin{theo}[Equivariant loop theorem for Kleinian groups] \label{lifting}
Let $K$ be a Kleinian group with region of discontinuity $\Omega \neq \emptyset$, let $\Delta$ be a (non-empty) collection of connected components of $\Omega$ which is invariant under the action of $K$, let ${\mathcal O}=\Delta/K$, let ${\mathcal B}$ be the cone points of ${\mathcal O}$ and  let $H<{\rm Aut}({\mathcal O})$ be a finite group of automorphisms of ${\mathcal O}$. Let us assume that ${\mathcal O}$ consists of (may be infinitely many) analytically finite Riemann orbifolds.
Fix some regular (branched) covering map $\pi:\Delta \to {\mathcal O}$ with $K$ as its deck group.
Let ${\mathcal G}$ be the collection of loops in ${\mathcal O}-{\mathcal B}$ which lift, with respect to $\pi$, to loops in $\Delta$. If $H$ lifts to a group of automorphisms of $\Delta$, then there is a sub-collection ${\mathcal F} \subset {\mathcal G}$ such that:
\begin{enumerate}
\item ${\mathcal F}$ consists of pairwise disjoint  powers of simple loops;
\item ${\mathcal F}$ is $H$-invariant; and
\item every loop in ${\mathcal G}$ is homotopic to the product of finite powers of a finite sub-collection of loops in ${\mathcal F}$.
\end{enumerate}
\end{theo}
\begin{proof}
Let us consider a maximal subcollection of non-equivalent components of $\Delta$ under the action of $K$, say $\Delta_{j}$ for $j \in J$. Let $K_{j}$ be the $K$-stabilizer of $\Delta_{j}$ under the action of $K$. By Theorem \ref{lifting2}, on ${\mathcal O}_{j}=\Delta_{j}/K_{j}$ there is a collection of loops, say ${\mathcal F}_{j}$, satisfying the properties on that theorem. Clearly the collection of fundametal loops ${\mathcal F}=\cup_{j \in J} {\mathcal F}_{j}$ is the required one.
\end{proof}

\begin{rema}
The condition for ${\mathcal O}=\Delta/K$ to consists of analytically finite Riemann orbifolds is equivalent, by the Ahlfors finiteness theorem, for the $K$-stabilizer of each connected component in $\Delta$ to be finitely generated. In particular, if $K$ is finitely generated, then ${\mathcal O}$ is a finite collection of analytically finite Riemann surfaces and ${\mathcal F}$ turns out to be a finite collection. If, in Theorem \ref{lifting}, we  assume $K$ to be torsion-free, then the loops in ${\mathcal F}$ will be simple loops. 
\end{rema}

\subsection{A connection to Klenian $3$-manifolds}
Let $K$ be a Kleinian group, with region of discontinuity $\Omega \subset \widehat{\mathbb C}$. There is a natural discrete action (by Poincar\'e extension) of $K$ on the upper half-space 
${\mathbb H}^{3}=\{(z,t): z \in {\mathbb C}, \; t \in (0,+\infty)\}$, which is given  by isometries in the hyperbolic metric $ds^{2}=(|dz|^{2}+dt^{2})/t^{2}$. The quotient $M_{K}=({\mathbb H}^{3} \cup \Omega)/K$ carries the structure of a $3$-orbifold, its interior ${\mathbb H}^{3}/K$ an structure of a complete hyperbolic $3$-orbifold and $\Omega/K$ the structure of a Riemann orbifold. In the case that $K$ is torsion free, all the above are manifolds and we say that $M_{K}$ is a Kleinian $3$-manifold.

A direct consequence of Theorem \ref{lifting} is the equivariant theorem for Kleinian $3$-manifolds in the case that the conformal boundary is non-empty and it consists of analytically finite Riemann surfaces.

\begin{coro}\label{corolario1}
Let $K$ be a torsion free Kleinian group, with non-empty region of discontinuity $\Omega$, such that $S_{K}=\Omega/K$ is a collection (it might be infinitely many of them) of analytically finite Riemann surfaces. Let $H$ be a finite group of automorphismsm of the Kleinian $3$-manifold $M_{K}=({\mathbb H}^{3} \cup \Omega)/K$.
 If ${\mathcal G}$ is the collection of loops on $S_{K}$ that are homotopically nontrivial in $S_{K}$ but homotopically trivial in $M_{K}$, then there exists a collection of pairwise disjoint simple loops  ${\mathcal F} \subset {\mathcal G}$, equivariant under the action of $H$,  so that ${\mathcal G}$ is the smallest normal subgroup of $\pi_{1}(S_{K})$ generated by ${\mathcal F}$.
\end{coro}

\begin{rema}
Let $K$ be a torsion free Kleinian group and let $H$ be as in Corollary \ref{corolario1}. Then the following hold.
(1) If $\pi_{1}(M)$ is finitely generated, then the collection ${\mathcal F}$ is finite.
(2) By lifting $H$ to the universal cover space, one obtains a (extended) Kleinian group $\widehat{K}$ containing $K$ as a finite index normal subgroup so that $H=\widehat{K}/K$. Corollary \ref{corolario1} may be used to obtain a geometric structure picture of $\widehat{K}$, in the sense of the Klein-Maskit combination theorems, in terms of the algebraic structure of $H$.
(3) If $M_{K}$ is compact, then the result follows from Meeks-Yau's equivariant loop theorem \cite{Y-M1,Y-M2}, whose arguments are based on minimal surfaces theory.  If $K$ is not a purely loxodromic geometrically finite Keinian group, then $M_{K}$ is non-compact and the result is not longer a consequence of Meek's-Yau's equivariant theorem.
\end{rema}

\section{Proof of Theorem \ref{corolario3}}\label{Sec:Klein-Maskit-Comb}
We extend Maskit's decomposition theorems to the theorems below, to apply to extended groups.

The proof of Theorem \ref{corolario3} is a direct consequence of Theorem \ref{corolario2}, which is the main step, and Theorem \ref{propo} as described below. 
If the word ``extended" is removed, the statements of these theorems are simply Maskit's original theorems (see \cite{Maskit:construction, Maskit:function2, Maskit:function3, Maskit:function4}]).

\begin{theo}[First step in Maskit-type decomposition of an extended function groups]\label{corolario2}
Let $K$ be an extended function group. Then, 
$K$ is constructed, using the Klein-Maskit combination theorems, as amalgamated free products and HNN-extensions using a finite collection of (extended) $B$-groups. Moreover, the amalgamations and HNN-extensions are realized along either trivial or a finite cyclic group or a dihedral group generated by two reflections (this last one only in the amalgamated free product operation). 
\end{theo}

\begin{theo}[Decomposition of extended B-groups]\label{propo}
Let $K$ be an extended B-group with a simply-connected invariant component $\Delta$. Then either (i) $K$ is an elementary extended Kleinian group or (ii) $K$ is an extended quasifuchsian group  or (iii) $K$ is an extended degenerate group or (iv) $\Delta$ is the only invariant component and $K$ is constructed as amalgamated free products and HNN-extensions, by use of the Klein-Maskit combination theorems, using (extended) elementary groups, (extended) quasifuchsian groups and (extended) totally degenerate groups. The 
amalgamated free products and HNN-extensions are given along axes of accidental parabolic transformations.
\end{theo}

\begin{rema}
We note for the reader that the proof of Theorem \ref{corolario2} includes Remarks \ref{estructuradeH} and \ref{observa15} and Lemmas \ref{lemita2} and \ref{lemita1} and that the proof of Theorem \ref{propo} includes Lemmas \ref{lema17} and \ref{lema2}.
\end{rema}

\subsection{Proof of Theorem \ref{corolario2}}\label{Sec:pruebacoro2}
Let $K$ be an extended function group and let $\Delta$ be a $K$-invariant connected component of its region of discontinuity (we may assume $K$ to be non-elementary). If there is another different invariant connected connected component of its region of discontinuity, then $K^{+}=K \cap {\mathbb M}$ is known to be a quasifuchsian group \cite{Maskit:extended}; so $K$ is an (extended) quasifuchsian group. Let us assume, from now on, that $\Delta$ is the unique invariant connected component. 
By Selberg's lemma \cite{Selberg}, there is a torsion free finite index normal subgroup $G_{1}$ of $K^{+}$. As $K=\langle K^{+},\tau\rangle$, where $\tau^{2} \in K^{+}$, one has that $G=G_{1} \cap \tau G_{1} \tau^{-1}$ is a torsion free finite index normal subgroup of $K$. 

It follows that $G$ is a function group with $\Delta$ as an invariant connected component of its region of discontinuity (the same as for $K$). Also, $\Delta$ is the only invariant connected component of $G$; otherwise $G$ is a quasifuchsian group and $K$ will have two different invariant connected components, which is a contradiction to our assumption on $K$.
If $K \neq K^{+}$, then we have that $G \neq K$. Now, 
if $K=K^{+}$, then we may also assume that $G \neq K$. In fact, if $K=K^{+}=G$, then we may consider the normal subgroup $G^{2}$, generated by all squares of elements of $G$. This is a normal subgroup different from $G$ (as $G$ is non-elementary, it contains simple loxodromic elements) and, as $G$ is finitely generated, this is of finite index.
Let $S=\Delta/G$ (an analytically finite Riemann surface) and consider a regular planar unbranched cover $P:\Delta \to S$ with $G$ as its deck group.
Set $H=K/G<Aut(S)$, which is a non-trivial  finite group as $G \neq K$.
Theorem \ref{lifting2} asserts the existence of a fundamental set of loops ${\mathcal F} \subset S$ for the pair $(G,H)$. Such a collection of loops cuts $S$ into some finite number of connected components and such a collection of components is invariant under $H$. The $H$-stabilizer of each of these connected components and each of the loops  in ${\mathcal F}$ is a finite group.

\begin{rema}[Decomposition structure of $H$]\label{estructuradeH}
The $H$-equivariant fundamental system of loops ${\mathcal F}$ permits to obtain an 
structure of $H$ as a finite iteration of amalgamated free products and HNN-extensions of certain subgroups of $H$ as follows. Let us consider a maximal collection of components of $S-{\mathcal F}$, say $S_{1}$\ldots, $S_{n}$, so that any two different components are not $H$-equivalent. Let us denote by $H_{j}$ the $H$-stabilizer of $S_{j}$. It is possible to chose these surfaces so that, by adding some on the boundary loops, we obtain a planar surface $S^{*}$ (containng each $S_{j}$ in its interior). If two surfaces $S_{i}$ and $S_{j}$ have a common boundary in $S^{*}$, then $H_{i} \cap H_{j}$ is either trivial or a cyclic group (this being exactly the $H$-stabilizer of the common boundary loop).  We perform the amalgamated free product of  $H_{i}$ and $H_{j}$ along the trivial or cyclic group $H_{i} \cap H_{j}$. Set $S_{ij}$ be the union of $S_{i}$, $S_{j}$ with the common boundary loop in $S^{*}$ and set $H_{ij}$ the constructed group. Now, if $S_{k}$ is another of the surfaces which has a common boundary loop in $S^{*}$ with $S_{ij}$, then we again perform the amalgamated free product of  $H_{ij}$ and $H_{k}$ along the trivial or cyclic group $H_{ij} \cap H_{k}$.  Continuing with this procedure, we end with a group $H^{*}$ obtained as amalgamated free product along finite cyclic groups or trivial groups. For each boundary of $S^{*}$ we add a boundary loop, in order to stay with a planar compact surface (we are out of $S$ in this part). If $\alpha$ is any of the boundary loops of $S^{*}$, there should be 
another boundary loop $\beta$ of $S^{*}$ and an element $h \in H$ so that $h(\alpha)=\beta$. By the choice of the surfaces $S_{j}$, we must have that $h(S^{*}) \cap S^{*}=\emptyset$. In particular, $\beta \neq \alpha$. If there is another element $k \in H-\{h\}$ so that $k(\alpha)=\beta$, then $k^{-1}h$ is a non-trivial element that stabilizes $\alpha$ and $k^{-1}h(S^{*}) \cap S^{*} \neq \emptyset$, a contradiction. Also, if there is another boundary loop $\gamma$ of $S^{*}$ (different from $\beta$) and an element $u \in H$ so that $u(\alpha)=\gamma$, then $uh^{-1} \in H-\{I\}$ satisfies that $uh^{-1}(S^{*}) \cap S^{*} \neq \emptyset$, which is again a contradiction. We may now perform the HHN-extension of $H^{*}$ by the finite cyclic group generated by $h$. If $\alpha_{1}=\alpha$,\ldots, $\alpha_{m}$ are the boundary loops of $S^{*}$, which are not $H$-equivalent, then we perform the HHN-extension with each of them. At the end, we obtain an isomorphic copy of $H$.
\end{rema}

We may assume that the fundamental set of loops ${\mathcal F}$ to be minimal, that is, by deleting any non-empty subcollection of loops from it, then the obtained subcollection fails to be a fundamental set of loops for $(G,H)$. The minimality condition asserts that each connected component of $S-{\mathcal F}$ cannot be either a disc or an annulus.
By lifting ${\mathcal F}$ to $\Delta$, under $P$, one obtains a collection $\widehat{\mathcal F} \subset \Delta$ of pairwise disjoint simple loops, so that $\widehat{\mathcal F}$ is invariant under the group $K$.  Each of the loops in $\widehat{\mathcal F}$ is called a {\it structure loop} and each of the connected components of  $\Delta-\widehat{\mathcal F}$ a {\it structure region}. These structure loops and regions are permuted by the action of $K$. 
The $K$-stabilizer (respectively, the $G$-stabilizer) of each structure loop and each structure region is called a structure subgroup of $K$ (respectively, a structure subgroup of $G$).

If $R$ is a structure region, then its $K$-stabilizer, denoted by $K_{R}$, is a finite extension of its $G$-stabilizer, denoted by $G_{R}$. Similarly, if $\alpha$ is a structure loop, then its $K$-stabilizer is a finite extension of its $G$-stabilizer.

\begin{lemm}\label{lemita2}
Let $\alpha$ be a structure loop and let $R$ be a structure region containing $\alpha$ on its border. Then the $K_{R}$-stabilizer of $\alpha$ is either trivial or a finite cyclic group or a dihedral group generated by two reflections (both circles of fixed points intersecting at two points, one inside of one of the two discs  bounded by $\alpha$ and the other point contained inside the other disc). Moreover, 
the $K$-stabilizer of $\alpha$ is either equal to its $K_{R}$-stabilizer or it is generated by its $K_{R}$-stabilizer and an involution (conformal or anticonformal) that sends $R$ to the other structure region containing $\alpha$ in its border.
\end{lemm}
\begin{proof}
Let $\alpha \in \widehat{\mathcal F}$ be a structure loop. As $\alpha$ is contained in the region of discontinuity of $K$, the $K$-stabilizer of $\alpha$ is a finite group; so also its $G$-stabilizer is finite.  Note that the $K^{+}$-stabilizer of $\alpha$ is either trivial, finite cyclic group or a dihedral group. Moreover, in the dihedral case, one of the involutions interchanges both discs bounded by $\alpha$. Let $R$ be a structure region containing $\alpha$ as a boundary loop. Then the $K^{+}_{R}$-stabilizer of $\alpha$ is either trivial or a finite cyclic group. It follows that the $K_{R}$-stabilizer of $\alpha$ is either trivial or a finite cyclic group or a dihedral group generated by two reflections (both circles of fixed points intersecting at two points, one inside of one of the two discs  bounded by $\alpha$ and the other point contained inside the other disc).  The $K$-stabilizer of $\alpha$ is generated by the $K_{R}$-stabilizer and probably an extra involution (conformal or anticonformal) that interchanges both discs bounded by $\alpha$.
\end{proof}

\begin{rema}\label{observa15}
We do not need this extra information for the rest of the proof, but it may help with a clarification of the gluing process at the Klein-Maskit combination theorems.
It follows, from Lemma \ref{lemita2}, that the $K$-stabilizer of $\alpha \in \widehat{\mathcal F}$ must be one of the followings: 
(1) the trivial group, 
(2) a cyclic group generated by a reflection with $\alpha$ as its circle of fixed points (so it permutes both discs bounded by $\alpha$), 
(3) a cyclic group generated by a reflection that keeps invariant each of the two discs bounded by $\alpha$ (the reflection has exactly two fixed points over $\alpha$), 
(4) a cyclic group generated by an imaginary reflection (it permutes both discs bounded by $\alpha$), 
(5) a cyclic group generated by an elliptic transformation of order two (permuting the two discs bounded by $\alpha$), 
(6) a cyclic group generated by an elliptic transformation (preserving each of the two discs bounded by $\alpha$), 
(7) a group generated by an elliptic transformation (preserving each of the two discs bounded by $\alpha$) and a reflection whose circle of fixed points is $\alpha$,
(8) a group generated by an elliptic transformation (preserving each of the two discs bounded by $\alpha$) and an imaginary reflection (permuting both discs bounded by $\alpha$),
(9) a group generated by an elliptic transformation of order two (permuting the two discs bounded by $\alpha$) and an imaginary reflection that keeps $\alpha$ invariant (it permutes both discs bounded by $\alpha$),
(10) a dihedral group generated by two reflections (both circles of fixed points intersecting at two points, one inside of one of the two discs  bounded by $\alpha$ and the other point contained inside the other disc), 
(11) a group generated by an elliptic transformation (preserving each of the two discs bounded by $\alpha$) and an imaginary reflection that keeps $\alpha$ invariant (it permutes both discs bounded by $\alpha$),
(12) a group generated  by a dihedral group of M\"obius transformations and a reflection with $\alpha$ as circle of fixed points,
(13) a group generated by a dihedral group generated by two reflections (both circles of fixed points intersecting at two points, one inside of one of the two discs  bounded by $\alpha$ and the other point contained inside the other disc) and an elliptic transformation of order two that permutes both discs bounded by $\alpha$,
To obtain the above, we use the following fact. Let  $\alpha$ be a loop which is invariant under (i) an elliptic transformation $E$, of order two that interchanges both discs bounded by it, and (ii) also invariant under an imaginary reflection $\tau$. Then $E \tau$ is necessarily a reflection whose circle of fixed points is transversal to $\alpha$.

\end{rema}

Let $R$ be a structure region and let $\alpha \in \widehat{\mathcal F}$ be on the boundary of $R$. By Lemma \ref{lemita2}, the $K_{R}$-stabilizer of $\alpha$ is some finite group; either trivial or a finite cyclic group or a dihedral group generated by two reflections (both circles of fixed points intersecting at two points, one inside of one of the two discs  bounded by $\alpha$ and the other point contained inside the other disc).
Let $D_{\alpha}$ be the topological disc bounded by $\alpha$ and disjoint from $R$. Clearly, the $K_{R}$-stabilizer of such a disc is contained in the $K_{R}$-stabilizer of $\alpha$ (each element of $K_{R}$ that stabilizes $D_{\alpha}$ also stabilizes $\alpha$), so $D_{\alpha}$ is contained in the region of discontinuity of $K_{R}$. 
It follows that $K_{R}$ is a (extended) function group with an invariant connected component $\Delta_{R}$ of its region of discontinuity containing $R$ and all the discs $D_{\alpha}$, for every structure loop $\alpha$ on its boundary.

\begin{lemm}\label{lemita1}
$\Delta_{R}$ is simply-connected. 
\end{lemm}
\begin{proof}
If $\Delta_{R}$ is not simply-connected, then there is a simple loop $\beta \subset R$ bounding two topological discs, each one containing limit points of $K_{R}$ (so limit points of $K$). The projection on $S$ of $\beta$  produces a loop $\widetilde{\beta} \subset S$ which lifts to a loop under $P$. But, we know that $\widetilde{\beta}$ is homotopic to the product of finite powers of the simple loops on the boundary of the finite domain $P(R) \subset S$. It follows that $\beta$ must be homotopic to the product of finite powers of a finite collection of structure loops on the boundary of $R$. As each of these boundary loops bounds a disc containing no limit points, we get a contradiction for $\beta$ to bound two discs, each one containing limit points.
\end{proof}

We may follow the same lines as described in Remark \ref{estructuradeH} to obtain that $K$ is constructed, using the Klein-Maskit combination theorems \cite{Maskit:book,Maskit:Comb}, as amalgamated free products and HNN-extensions using a finite collection of the structure subgroups of $K_{R}$
(which, by Lemma \ref{lemita1}, are extended B-groups with invariant simply-connected component $\Delta_{R}$).  By Lemma \ref{lemita2}, the amalgamations and HNN-extensions are realized along either trivial or a finite cyclic group or a dihedral group generated by two reflections. This ends the proof of Theorem \ref{corolario2}. 
\hfill $\square$

\subsection{Proof of Theorem \ref{propo}}\label{Sec:prueba3}
We proceed to describe the subtle modifications in  Maskit's arguments in the decomposition of B-groups  \cite{Maskit:function3, Maskit:function4}
adapted to the case of extended B-groups (see also chapter IX.H. in \cite{Maskit:book}). Let us assume that $K$ is an extended B-group and that it is neither a (extended) elementary group or a (extended) quasifuchsian group or a (extended) degenerate group. 
Let $\Delta$ be the simply-connected invariant component of the region of discontinuity of  $K$.  Every other connected component of the region of discontinuity of $K$ is simply-connected (see Proposition IX.D.2. in \cite{Maskit:book}). By our assumptions on $K$, we have that $K^{+}$ is neither elementary nor degenerate Kleinian group. It may be, even if $K$ is not an extended quasifuchsian, that $K^{+}$ is a quasifuchsian. But in this case, we have that $K$ is just a HNN-extension of a quasifuchsian group along a cyclic group. So, from now on, we assume that $K^{+}$ is neither a quasifuchsian group.
By the Klein-Poincar\'e uniformization theorem \cite{Poincare} and the fact that $K$ is non-elementary, $\Delta$ 
is isomorphic to upper half-plane ${\mathbb H}^{2}$. Let us consider a bi-holomorphism $f:{\mathbb H}^{2} \to \Delta$. The group $f^{-1} K f$ is a group of conformal and anticonformal automorphisms of ${\mathbb H}^{2}$, in particular, an extended B-group with ${\mathbb H}^{2}$ as an invariant connected component of its region of discontinuity. In this case, $f^{-1} K^{+} f$ is a co-finite fuchsian group, that is, ${\mathbb H}^{2}/f^{-1} K^{+} f$ has finite hyperbolic area. It is known that $f$ sends parabolic transformations to parabolic transformations, but it may send a hyperbolic transformation to a parabolic one.
A parabolic element $P \in K^{+}$ is called {\it accidental} if $f^{-1} P f$ is a hyperbolic transformation. In this case, the image under $f$ of the axis of the hyperbolic transformation $f^{-1} P f$ is called the {\it axis} of $P$ (in Maskit's notation this is the true axis of $P$). 
As it is well known that no rank two parabolic subgroup can preserve a disc in the Riemann sphere, it follows that $f^{-1} K^{+} f$ does not contains  rank two parabolic subgroup, in particular, $K^{+}$ neither does contains a rank two parabolic subgroup. 
Theorem IX.D.21 in \cite{Maskit:book} states that $K^{+}$ is either quasifuchisian or totally degenerate or it contains accidental parabolics. By our assumptions on $K$ and $K^{+}$, we have that necessarily 
$K^{+}$ must have accidental parabolic transformations.  Moreover, there is a finite number of conjugacy classes of primitive accidental parabolic transformations in $K^{+}$. Let us consider a collection of accidental parabolic transformations in $K^{+}$, say $P_{1}$,..., $P_{m}$, so that $P_{j}$ is not $K^{+}$-conjugate to $P_{r}^{\pm 1}$ if $j \neq r$, and $P_{j}$ is primitive, that is it is not of the form $Q^{a}$ for some $Q \in K$ and $a \geq 2$. Let us denote by $L_{j} \subset \Delta$ the axis of 
$P_{j}$ (note that $L_{j}$ is a geodesic for the hyperbolic metric of $\Delta$ and that $P_{j}$ keeps it invariant acting by a translation on it).

\begin{lemm}\label{lema17}
(1) If $j \neq r$, then the $K^{+}$-translates of $L_{j}$ do not intersect the $K^{+}$-translates of $L_{r}$. 
(2) For each fixed $j$, any $K^{+}$-translates of $L_{j}$ is either disjoint from $L_{j}$ or it coincides with it.
\end{lemm}
\begin{proof}
Let us consider a Riemann map $f:{\mathbb H}^{2} \to \Delta$, where ${\mathbb H}^{2}$ is the upper half-plane with the hyperbolic metric $ds^{2}=|dz|^{2}/{\rm Im}(z)^{2}$. It is well known that any two different geodesics in ${\mathbb H}^{2}$ are either disjoint of they intersect at exactly one point. The push-forward of the hyperbolic metric in ${\mathbb H}^{2}$ provides the hyperbolic metric of $\Delta$. It follows that any $K^{+}$-translate of $L_{j}$ and any $K^{+}$ translate of $L_{r}$ (for $j$ not necessarily different from $r$) are either disjoint or they intersect exactly at one point or they are the same. 
Let us first prove (1), that is, we assume $j \neq r$.
If there are $K^{+}$-translates of $L_{j}$ and $L_{r}$ which are the same, as $P_{j}$ and $P_{r}$ are primitive parabolic, share the same fixed point and $K^{+}$ is discrete, then $P_{j}$ is conjugate to either $P_{r}^{\pm 1}$, a contradiction. If there are $K^{+}$-translates of $L_{j}$ and $L_{r}$ which intersect at a point, then the planarity of $\Delta$ asserts that the non-empty intersection only may happens if a $K^{+}$ conjugate of $P_{j}$ and a $K^{+}$-conjugate of $P_{r}$ share their unique fixed point. The discreteness of $K^{+}$ asserts that $K^{+}$ must contains a rank two parabolic subgroup, a contradiction.
Let us now prove (2), that is, we assume $j=r$. This follows the same lines a the previous case to see that either the translates are either disjoint or equal. 
\end{proof}

\begin{lemm}\label{lema2}
If $T \in K-K^{+}$, then $T$ preserves the collection of $K^{+}$-translates of $\{L_{1},....,L_{m}\}$.
\end{lemm}
\begin{proof}
$T$ acts as an isometry on $\Delta$ and must permute the accidental parabolic transformations. As the axis are unique for each accidental parabolic, we are done.
\end{proof}

Let $\widehat{L}_{j}$ be equal to $L_{j}$ together the corresponding fixed point of $P_{j}$. Then the collection ${\mathcal F}$ given by the  $K^{+}$-translates of $\{\widehat{L}_{1},..., \widehat{L}_{m}\}$ consists of pairwise disjoint simple loops; each one is called a {\it structure loop} for the group $K$. Such a collection of structure loops still invariant for any $T \in K-K^{+}$ by Lemma \ref{lema2}. The structure loops cut $\Omega$ (the region of discontinuity of $K$) and $\Delta$ into regions; called {\it structure regions} for $K$. These are different from our previous definitions of structure loops and regions as these ones are not completely contained in the region of discontinuity.

Let $\alpha \in {\mathcal F}$ be a structure loop and let $R_{1}$ and $R_{2}$ be the two structure regions containing $\alpha$ in their common boundary. Let $K_{j}<K$ be the $K$-stabilizer of $R_{j}$, 
let $K_{\alpha}$ be the $K$-stabilizer of $\alpha$ and let $P \in K$ be the primitive accidental parabolic transformation whose axis is $\alpha$ (which is then $K$-conjugated to some of the $P_{j}$'s). Clearly, $\langle P \rangle$ is contained in $K_{j}$, 
$\langle P \rangle < K_{\alpha}$ and either (i) $\langle P \rangle = K_{\alpha}$ or (ii) $\langle P \rangle$ has index two in $K_{\alpha}$ or (iii) 
$\langle P \rangle$ has index four in $K_{\alpha}$ (this last case means that $\langle P \rangle$ has index two inside the $K_{j}$-stabilizer of $\alpha$). 
The region $R_{3-j}$ is contained in a disc $D_{3-j}$, whose $K_{j}$-stabilizer is equal to the $K_{j}$-stabilizer of the loop $\alpha$; this is either the cyclic group generated by $P$ or it contains it as an index two subgroup. It follows that $D_{3-j}$ is contained in the region of discontinuity $\Omega_{j}$ of $K_{j}$ and that there is an invariant connected component $\Delta_{j} \subset \Omega_{j}$ containing $\Delta$. 
Lemma IX.H.10 in \cite{Maskit:book} states that $K_{j}^{+}$ is a B-group, with $\Delta_{j}$ as invariant simply-connected component, without accidental parabolic transformations. It follows that $K_{j}^{+}$ is either elementary or quasifuchsian or totally degenerate, in particular, that $K_{j}$ is either (extended) elementary or (extended) quasifuchsian or (extended) totally degenerate. 
One possibility is that $K_{\alpha}$ is an extension of degree two of the $K_{j}$-stabilizer of $\alpha$. In this case, there is an element $Q \in K_{\alpha}$ that permutes $R_{1}$ with $R_{2}$ ($Q$ is either a pseudo-parabolic whose square is $P$ or an involution). In this case, $\langle K_{1},K_{2}\rangle$ is the HNN-extension of $K_{1}$ by $Q$ (in the sense of the second Klein-Maskit combination theorem). The other possibility is that $K_{\alpha}$ is equal to $K_{1} \cap K_{2}$ (either the cyclic group generated by the parabolic $P$ or a group generated by two reflections sharing as a common fixed point the fixed point of $P$). In this case, $\langle K_{1},K_{2}\rangle$ is the free product of $K_{1}$ and $K_{2}$ amalgamated over $K_{1} \cap K_{2}$ (in the sense of the first Klein-Maskit combination theorem).

Now, following the same ideas in \cite{Maskit:function3, Maskit:function4}, one obtains a decomposition of $K$ as an amalgamated free products and HNN-extensions, by use of the Klein-Maskit combination theorems, using (extended) elementary groups, (extended) quasifuchsian groups and (extended) totally degenerate groups. 
 \hfill $\square$

\subsection*{Acknowledgment}
The author would like to thank the referees for their valuable comments, suggestions and corrections.



\begin{thebibliography}{99}


\bibitem{Ahlfors}
Ahlfors, L. V. 
Finitely generated Kleinian groups.
{\it Amer. J. of Math.} {\bf 86} (1964), 413--429;
doi.org/10.2307/2373173

\bibitem{Ahlfors2}
Ahlfors, L. V.
Correction to ``Finitely generated Kleinian groups".
{\it  American Journal of Mathematics} {\bf 87} (1965), 759;
doi.org/10.2307/2373073 

\bibitem{Haas}
Haas. A.
Linearization and Mappings onto Pseudocircle Domains.
{\it Trans. of the Amer. Math. Soc.} {\bf 282}  No. 1 (1984), 415--429;
doi.org/10.2307/1999596

\bibitem{H-M:lifting}
Hidalgo, R. A. and Maskit, B.
A Note on the Lifting of Automorphisms. 
In Geometry of Riemann Surfaces. 
{\it Lecture Notes of the London Mathematics Society} {\bf 368}, 2009.
Edited by Fred Gehring, Gabino Gonzalez and Christos Kourouniotis.
ISBN: 978-0-521-73307-6; 
doi.org/10.1017/cbo9781139194266.013 

\bibitem{Maskit:planarcover}
Maskit, B.
A theorem on planar covering surfaces with applications to $3$-manifolds.
{\it Annals of Math.} {\bf 81} No.2 (1965), 341--355;
doi.org/10.2307/1970619

\bibitem{Maskit:construction}
Maskit, B.
Construction of Kleinian groups.
{\it Proc. of the conf. on Complex Anal.}, Minneapolis, 1964, Springer-Verlag, 1965;
doi.org/10.1007/978-3-642-48016-4$\textunderscore$24 

\bibitem{Maskit:Comb}
Maskit, B.
On Klein's combination theorem III.
Advances in the Theory of Riemann Surfaces (Proc. Conf., Stony Brook, N.Y., 1969),
{\it  Ann. of Math. Studies} {\bf 66} (1971), 
Princeton Univ. Press, 297-316;
doi.org/10.2307/j.ctt9qh044

\bibitem{Maskit:Comb4}
Maskit, B.
On Klein's combination theorem. IV. 
{\it Trans. Amer. Math. Soc.} {\bf 336} (1993), 265--294;
doi.org/10.2307/2154347


\bibitem{Maskit:function2}
Maskit, B.
Decomposition of certain Kleinian groups.
{\it Acta Math.} {\bf 130} (1973), 243--263;
doi.org/10.1007/bf02392267 

\bibitem{Maskit:function3}
Maskit, B. 
On the classification of Kleinian Groups I. Koebe groups. 
{\it Acta Math.} {\bf 135} (1975), 249--270;
doi.org/10.1007/bf02392021 

\bibitem{Maskit:function4}
Maskit, B. 
On the classification of Kleinian Groups II. Signatures. 
{\it Acta Math.} {\bf 138} (1976), 17--42;
doi.org/10.1007/bf02392312 

\bibitem{Maskit:extended}
Maskit, B.
On extended quasifuchsian groups. 
{\it Ann. Acad. Sci. Fenn. Ser. A I Math.} {\bf 15} No. 1 (1990),  53--64;
doi.org/10.5186/aasfm.1990.1523 

      
\bibitem{Maskit:book}
Maskit, B.
{\it Kleinian Groups},
GMW, Springer-Verlag, 1987; 
doi.org/10.1007/978-3-642-61590-0 

\bibitem{MT}
Matsuzaki, K. and Taniguchi, M.
{\it Hyperbolic Manifolds and Kleinian Groups}.
Oxford Mathematical Monographs. Oxford Science Publications. The Clarendon Press, Oxford University Press, New York, 1998.



\bibitem{Y-M1}
Meeks, W. and Yau, S-T.
The equivariant Dehn's lemma and loop theorem.
{\it Comment. Math. Helvetici} {\bf 56} (1981), 225--239;
doi.org/10.1007/bf02566211 

\bibitem{Y-M2}
Meeks, W. and Yau, S-T.
The equivariant loop theorem for three-dimensional manifolds and a review of the existence theorems for minimal surfaces. 
The Smith conjecture (New York, 1979), 153--163, 
{\it Pure Appl. Math.}, {\bf 112}, Academic Press, Orlando, FL, 1984;
doi.org/10.1016/s0079-8169(08)61640-2 

\bibitem{Poincare0}
Poincar\'e, H.
Memoire: les groupes Klein\'eens.
{\it Acta Math.} {\bf 3} (1882), 193--294;
doi.org/10.1007/bf02422441 

\bibitem{Poincare}
Poincar\'e, H.
Sur l'uniformisation des fonctions analytiques.
{\it Acta Math.} {\bf 31} (1907), 1--6; 
doi.org/10.1007/bf02415442 

\bibitem{Selberg}
Selberg, A. 
On discontinuous groups in higher dimensional symmetric spaces. 
In: Contributions to function theory (Internat. Colloq. Function Theory, Bombay), 1960, 147--164



\end{thebibliography}
\end{document}